\documentclass[a4paper,11pt]{article}
\usepackage{graphicx}
\usepackage{epstopdf}
\usepackage{amsthm}
\usepackage[utf8]{inputenc}
\usepackage[english]{babel}
\newtheorem{theorem}{Theorem}[section]
\newtheorem{corollary}{Corollary}[theorem]

\theoremstyle{definition}

\usepackage{amssymb, amsmath}
\title{Integral operator on certain subclass of analytic function with negative coefficients}
\date{}
\begin{document}
\numberwithin{equation}{section}
\maketitle
\date{}
\begin{center}
\author{{\bf{G. M. Birajdar}}\vspace{.31cm}\\
School of Mathematics \& Statistics,\\
Dr. Vishwanath Karad MIT World Peace University,\\
 Pune (M.S) India 411038\\
Email: gmbirajdar28@gmail.com}\vspace{0.31cm}\\
\author{{\bf{N. D. Sangle}}\vspace{0.31cm}\\
Department of Mathematics,\\
D. Y. Patil College of Engineering \& Technology, Kolhapur\\
 (M.S.) India 416006\\
Email: navneet\_sangle@rediffmail.com}\vspace{.5cm}\\

\end{center}
\vspace{1cm}
\abstract{}
In this paper, we study subclass of analytic function with negative coefficient defined by integral operator in the unit disc $U = \left\{ {z \in C:\left| z \right| < 1} \right\}$. The results are included coefficient estimates, closure theorem and distortion theorems of functions belonging to this subclass. Also, we presented detailed study of uniformly convex and uniformly starlike functions. \\



\section{Introduction} 
\renewenvironment{proof}{{\bfseries Proof:}}{}
Let $A_j$ denote the class of functions of the form 
\begin{equation}
f(z) = z + \sum\limits_{k = j + 1}^\infty  {{a_k}{z^k}\,\,\,\left( {j \in N = \left\{ {1,2,3,...} \right\}} \right)}
\end{equation}
which are analytic in the unit disc $U = \left\{ {z \in C:\left| z \right| < 1} \right\}$.\\
The integral operator $I^{n}$ is defined in [1] by \\
\begin{align*}
{I^0}\,f(z) = f(z).
\end{align*}
\begin{align*}
{I^1}\,f(z) = I(z)=\int\limits_0^z {f(t){t^{ - 1}}dt;}
\end{align*}
\begin{align*}
{I^n}f(z) = I\left( {{I^{n - 1}}f(z)} \right),\,n \in N = \left\{ {1,2,3,...} \right\} 
\end{align*}
Integral operator for $f(z)$ is defined as:
\begin{align}
{I^n}f(z) = z + \sum\limits_{k = 2}^\infty  {{k^{ - n}}} {a_k}{z^k}
\end{align}
 
Using above operator ${I^n}$, we say that a function $f(z)$ belongs to
${A_j}$  is in  $S\left( {n,m,\beta} \right)$ if and only if 
\begin{align*}
{\mathop{\rm Re}\nolimits} \left\{ {\frac{{{I^{n + m}}f(z)}}{{{I^n}f(z)}}} \right\} \ge \beta \left| {\frac{{{I^{n + m}}f(z)}}{{{I^n}f(z)}} - 1} \right|
\end{align*}
for some $\beta \ge0$ and for all $z \in U$.\\
Let $T_j$ denote the subclass of $A_j$ consisting of functions of the form
\begin{equation}
f(z) = z - \sum\limits_{k = j + 1}^\infty  {{a_k}{z^k}\,\,\,\left(a_k\ge0  , {j \in N = \left\{ {1,2,3,...} \right\}} \right)}
\end{equation}
We define $T(n,m,\beta) = S(n,m,\beta) \cap {T_j}$.\\
 
The class of analytic function with negative coefficients have been studied by various researchers (\cite{B2},\cite{B3},\cite{B4},\cite{B5},\cite{B7},\cite{B8},\cite{B10}) and among are Robertson \cite{B6},  Sangle and Birajdar \cite{ B12}, few to mention.
\section{Main Results} 
In this section, we present some important results for the class.
\begin{theorem}
Let the function $f(z)$ be defined by (1.3) then  $f(z)$ belongs to $T(n,m,\beta)$ if and only if
\begin{equation}
\sum\limits_{k = j + 1}^\infty  {{{\left[ k \right]}^{ - n}}\left[ {\left( {\beta  + 1} \right){{\left[ k \right]}^{ - m}} - \beta } \right]} \,{a_k} \le 1.
\end{equation} 
The  result is sharp.
\end{theorem}
\begin{proof}
Assume that $f(z) \in T(n,m,\beta)$, then by definition 
\begin{align*}
{\mathop{\rm Re}\nolimits} \left\{ {\frac{{{I^{n + m}}f(z)}}{{{I^n}f(z)}}} \right\} \ge \beta \left| {\frac{{{I^{n + m}}f(z)}}{{{I^n}f(z)}} - 1} \right|, \quad z \in U.
\end{align*}
Equivalently,
\begin{align*}
{\mathop{\rm Re}\nolimits} \left\{ {\frac{{1 - \sum\limits_{k = j + 1}^\infty  {{{\left[ k \right]}^{ - n - m}}} \,{a_k}{z^{k - 1}}}}{{1 - \sum\limits_{k = j + 1}^\infty  {{{\left[ k \right]}^{ - n}}} \,{a_k}{z^{k - 1}}}}} \right\} \ge \beta \left| {\frac{{1 - \sum\limits_{k = j + 1}^\infty  {{{\left[ k \right]}^{ - n - m}}} \,{a_k}{z^{k - 1}}}}{{1 - \sum\limits_{k = j + 1}^\infty  {{{\left[ k \right]}^{ - n}}} \,{a_k}{z^{k - 1}}}} - 1} \right|
\end{align*}
\begin{align}
=\left| {\frac{{1 - \sum\limits_{k = j + 1}^\infty  {{{\left[ k \right]}^{ - n - m}}} \,{a_k}{z^{k - 1}} - \sum\limits_{k = j + 1}^\infty  {{{\left[ k \right]}^{ - n}}} \,{a_k}{z^{k - 1}}}}{{1 - \sum\limits_{k = j + 1}^\infty  {{{\left[ k \right]}^{ - n}}} \,{a_k}{z^{k - 1}}}}} \right|
\end{align}
Choosing value of $z$ on real axis so that left side of (2.2) is real and letting 
$z \to 1$, we get
\begin{align*}
\left[ {1 - \sum\limits_{k = j + 1}^\infty  {{{\left[ k \right]}^{ - n - m}}} \,{a_k}} \right] \ge \beta \sum\limits_{k = j + 1}^\infty  {\left[ {{{\left[ k \right]}^{ - n - m}} - {{\left[ k \right]}^{ - n}}} \right]} 
\end{align*}
which yields, 
\begin{align*}
\sum\limits_{k = j + 1}^\infty  {{{\left[ k \right]}^{ - n}}} \,\left[ {\left( {\beta  + 1} \right){{\left[ k \right]}^{ - n}} - \beta } \right]{a_k} \le 1.
\end{align*}
Conversely, suppose that (2.1) is true for $z \in U$, then
\begin{align*}
{\mathop{\rm Re}\nolimits} \left\{ {\frac{{{I^{n + m}}f(z)}}{{{I^n}f(z)}}} \right\} - \beta \left| {\frac{{{I^{n + m}}f(z)}}{{{I^n}f(z)}} - 1} \right| \ge 0
\end{align*}
\begin{align*}
\left[ {1 - \sum\limits_{k = j + 1}^\infty  {{{\left[ k \right]}^{ - n - m}}} \,{a_k}} \right] \ge \beta \sum\limits_{k = j + 1}^\infty  {\left[ {{{\left[ k \right]}^{ - n - m}} - {{\left[ k \right]}^{ - n}}} \right]{a_k}} \,.
\end{align*}
If
\begin{align*}
\left\{ {\frac{{1 - \sum\limits_{k = j + 1}^\infty  {{{\left[ k \right]}^{ - n - m}}} \,{a_k}{{\left| z \right|}^{k - 1}}}}{{1 - \sum\limits_{k = j + 1}^\infty  {{{\left[ k \right]}^{ - n}}} \,{a_k}{{\left| z \right|}^{k - 1}}}}} \right\} - \beta \left\{ {\frac{{\sum\limits_{k = j + 1}^\infty  {{{\left[ k \right]}^{ - n}}} \left[ {{{\left[ k \right]}^{ - m}} - 1} \right]\,{a_k}{{\left| z \right|}^{k - 1}}}}{{1 - \sum\limits_{k = j + 1}^\infty  {{{\left[ k \right]}^{ - n}}} \,{a_k}{{\left| z \right|}^{k - 1}}}}} \right\} \ge 0.
\end{align*}
That is, if
\begin{align*}
\sum\limits_{k = j + 1}^\infty  {{{\left[ k \right]}^{ - n}}} \left[ {\left( {\beta  + 1} \right){{\left[ k \right]}^{ - m}} - \beta } \right]{a_k} \le 1.
\end{align*}
Which completes the proof of the theorem. 
\end{proof}
\begin{corollary}
Let the function $f(z)$  defined by (1.3) is in the class $T(n,m,\beta)$ then
\begin{align*}
0 \le {a_k} \le \frac{1}{{\left[ {{k^{ - n}}} \right]\left[ {\left( {\beta  + 1} \right){k^{ - m}} - \beta } \right]}}\,\,,\quad k \ge j + 1.
\end{align*}
The result is sharp for the functions
\begin{align}
f(z) = z - \frac{1}{{\left[ {{k^{ - n}}} \right]\left[ {\left( {\beta  + 1} \right){k^{ - m}} - \beta } \right]}}.
\end{align}
\end{corollary}
\begin{theorem}
Let $0 \le {\beta _1} \le {\beta _2}$, then $T(n,m,{\beta _2}) \subseteq T(n,m,{\beta _1})$.
\end{theorem}
\begin{proof}
Let the function $f(z)$ be defined by (1.3) be in the class $T(n,m,\beta_2)$  then by Theorem 2.1, we have
\begin{align*}
\sum\limits_{k = j + 1}^\infty  {{{\left[ k \right]}^{ - n}}} \left[ {\left( {{\beta _2} + 1} \right){{\left[ k \right]}^{ - m}} - {\beta _2}} \right]\,{a_k} \le 1.
\end{align*}
Consequently, 
\begin{align*}
\sum\limits_{k = j + 1}^\infty  {{{\left[ k \right]}^{ - n}}} \left[ {\left( {{\beta _1} + 1} \right){{\left[ k \right]}^{ - m}} - {\beta _1}} \right]{a_k} \le \sum\limits_{k = j + 1}^\infty  {{{\left[ k \right]}^{ - n}}} \left[ {\left( {{\beta _2} + 1} \right){{\left[ k \right]}^{ - m}} - {\beta _2}} \right]\,{a_k}.
\end{align*} 

\end{proof}
\begin{theorem}
For $\beta \rightarrow 0$, $T(n + 1,m,\beta) \subseteq T(n,m,\beta)$.
\end{theorem}
\begin{proof}
Let the function $f(z)$ defined by (1.3) be in th class $T(n + 1,m,\beta)$
then by Theorem 2.1, we have 
\begin{equation*}
\sum\limits_{k = j + 1}^\infty  {{{\left[ k \right]}^{ - n - 1}}} \left[ {\left( {\beta  + 1} \right){{\left[ k \right]}^{ - m}} - \beta } \right]{a_k} \le 1.
\end{equation*} 
Consequently,
\begin{align*}
\sum\limits_{k = j + 1}^\infty  {{{\left[ k \right]}^{ - n}}} \left[ {\left( {\beta  + 1} \right){{\left[ k \right]}^{ - m}} - \beta } \right]{a_k} \le \sum\limits_{k = j + 1}^\infty  {{{\left[ k \right]}^{ - n - 1}}} \left[ {\left( {\beta  + 1} \right){{\left[ k \right]}^{ - m}} - \beta } \right]{a_k}. 
\end{align*} 
.
\end{proof}

\begin{theorem}
$T(n, m, \beta)$ is a convex set.
\end{theorem}
\begin{proof}
Let the function 
\begin{equation}
f(z) = z - \sum\limits_{k = j + 1}^\infty  {{a_{k,v}}} {z^k}\,\,\,\quad \left( {{a_{k,v}} \ge 0,\,v = 1,2} \right)
\end{equation}
be in the class $T(n,m,\beta)$. It is sufficient to show that $g(z)$ defined by 
\begin{align*}
g(z) = z - \sum\limits_{k = j + 1}^\infty  {\left[ {\lambda {a_{k,1}} + (1 - \lambda ){a_{k,2}}} \right]\,\,} {z^k}, \quad (0 \le \lambda  \le 1)
\end{align*}
is also in the class $T(n ,m,\beta)$.\\
By using Theorem 2.1, we obtain
\begin{equation*}
\sum\limits_{k = j + 1}^\infty  {{{\left[ k \right]}^{ - n}}\left[ {\left( {\beta  + 1} \right){{\left[ k \right]}^{ - m}} - \beta } \right]\left[ {\lambda {a_{k,1}} + \left( {1 - \lambda } \right){a_{k,2}}} \right] \le 1}. 
\end{equation*} 
which implies that $g(z) \in T(n, m,\beta)$.\\
Hence, $T(n ,m,\beta)$ is a convex set.
\end{proof}
\begin{theorem}
Let the function $f(z)$ be defined by (1.3) be in the class $T(n,m,\beta)$  then for $\left|z\right|=r<1$, 
\begin{equation}
\left| {{I^i}f(z)} \right| \ge r - \frac{{{r^{j + 1}}}}{{{{\left[ 2 \right]}^{ - n - i}}\left[ {\left( {\beta  + 1} \right){{\left[ 2 \right]}^{ - m}} - \beta } \right]}}
\end{equation}
and
\begin{equation}
\left| {{I^i}f(z)} \right| \le r + \frac{{{r^{j + 1}}}}{{{{\left[ 2 \right]}^{ - n - i}}\left[ {\left( {\beta  + 1} \right){{\left[ 2 \right]}^{ - m}} - \beta } \right]}}
\end{equation}
For $z \in U$ and \  $0\le i \le n$.
\end{theorem}
\begin{proof}
Note that $ f(z) \in T(n,m,\beta )$ if and only if 
${I^i}f(z) \in T(n-i,m,\beta)$ and
\begin{equation}
{I^i}f(z) = z - \sum\limits_{k = j + 1}^\infty  {{{\left[ k \right]}^{ - i}}{a_k}{z^k}} .
\end{equation}
By Theorem 2.1, we know that
\begin{align*}
{\left[ 2 \right]^{ - n - i}}\left[ {\left( {\beta  + 1} \right){{\left[ 2 \right]}^{ - m}} - \beta } \right]\sum\limits_{k = j + 1}^\infty  {{{\left[ k \right]}^{ - i}}{a_k}}  \le \sum\limits_{k = j + 1}^\infty  {{{\left[ k \right]}^{ - n}}\left[ {\left( {\beta  + 1} \right){{\left[ 2 \right]}^{ - m}} - \beta } \right]{a_k}}  \le 1.
\end{align*}
That is,\\
\begin{equation}
\sum\limits_{k = j + 1}^\infty  {{{\left[ k \right]}^{ - i}}{a_k} \le \frac{1}{{{{\left[ 2 \right]}^{ - n - i}}\left[ {\left( {\beta  + 1} \right){{\left[ 2 \right]}^{ - m}} - \beta } \right]}}} 
\end{equation}
\begin{align*}
\begin{array}{l}
\left| {{I^i}f(z)} \right| \le \left| z \right| + {r^{j + 1}}\sum\limits_{k = j + 1}^\infty  {{{\left[ k \right]}^{ - i}}{a_k}} \\
\,\,\,\,\,\,\,\,\,\,\,\,\,\,\, \le r + {r^{j + 1}}\frac{1}{{{{\left[ 2 \right]}^{ - n - i}}\left[ {\left( {\beta  + 1} \right){{\left[ 2 \right]}^{ - m}} - \beta } \right]}}
\end{array}
\end{align*}
and
\begin{align*}
\left| {{I^i}f(z)} \right| \ge r - {r^{j + 1}}\frac{1}{{{{\left[ 2 \right]}^{ - n - i}}\left[ {\left( {\beta  + 1} \right){{\left[ 2 \right]}^{ - m}} - \beta } \right]}}.
\end{align*}

\end{proof}
\begin{corollary}
Let the function $f(z)$ be defined by (1.3) be in the class $T(n,m,\beta)$  then for $\left|z\right|=r<1$,
\begin{equation}
\left| {f(z)} \right| \ge r - \frac{{{r^{j + 1}}}}{{{{\left[ 2 \right]}^{ - n}}\left[ {\left( {\beta  + 1} \right){{\left[ 2 \right]}^{ - m}} - \beta } \right]}}
\end{equation}
and
\begin{equation}
\left| {f(z)} \right| \le r + \frac{{{r^{j + 1}}}}{{{{\left[ 2 \right]}^{ - n}}\left[ {\left( {\beta  + 1} \right){{\left[ 2 \right]}^{ - m}} - \beta } \right]}},  \quad (z \in F).
\end{equation}
The equalities in (2.9) and (2.10) are attained for the function given by 
\begin{align*}
f(z) = z - \frac{{{z^{j + 1}}}}{{{{\left[ 2 \right]}^{ - n}}\left[ {\left( {\beta  + 1} \right){{\left[ 2 \right]}^{ - m}} - \beta } \right]}}.
\end{align*}
\end{corollary}
\begin{proof}
Taking $i=0$ in Theorem 2.5, we immediately obtain (2.9) and (2.10).
\end{proof}
\begin{theorem}
Let ${f_j}(0) = z$ and 
\begin{align*}
{f_k}(z) = z - \frac{1}{{{{\left[ k \right]}^{ - n}}\left[ {\left( {\beta  + 1} \right){{\left[ k \right]}^{ - m}} - \beta } \right]}}{z^k},\,\,\,\,\,\quad \left( {k \ge j + 1\,\,;\,\,n \in N} \right).
\end{align*}
For $\beta \ge 0$.Then $f(z)$ is in the class $T(n,m,\beta)$ if and only if it can be expressed as
\begin{align}
f(z) = \sum\limits_{k = j}^\infty  {{\mu _k}\,} {f_k}(z) \  where \  {\mu _k}\ge 0 \ \  and  \ \ \sum\limits_{k = j}^\infty  {{\mu _k}\,}  = 1.
\end{align}
\end{theorem}
\begin{proof}
Assume that \\
\begin{align*}
f(z) = \sum\limits_{k = j}^\infty  {{\mu _k}\,} {f_k}(z) = z - \sum\limits_{k = j + 1}^\infty  {\frac{1}{{{{\left[ k \right]}^{ - n}}\left[ {\left( {\beta  + 1} \right){{\left[ k \right]}^{ - m}} - \beta } \right]}}{z^k}} .
\end{align*}
Then it follows that, 
\begin{equation*}
\sum\limits_{k = j + 1}^\infty  {{{\left[ k \right]}^{ - n}}\left[ {\left( {\beta  + 1} \right){{\left[ k \right]}^{ - m}} - \beta } \right]} \frac{1}{{{{\left[ k \right]}^{ - n}}\left[ {\left( {\beta  + 1} \right){{\left[ k \right]}^{ - m}} - \beta } \right]}}{\mu _k} = \sum\limits_{k = j + 1}^\infty  {{\mu _k} = 1 - {\mu _j}}  \le 1.
\end{equation*}
Conversely, assume that the function defined by (1.3) belongs to class. Then
\begin{align*}
{a_k} \le \frac{1}{{{{\left[ k \right]}^{ - n}}\left[ {\left( {\beta  + 1} \right){{\left[ k \right]}^{ - m}} - \beta } \right]}},\,\,\quad (k \ge j + 1\,\,,n \in {N_0})
\end{align*} 
Setting,
\begin{align*}
{\mu _k} = {\left[ k \right]^{ - n}}\left[ {\left( {\beta  + 1} \right){{\left[ k \right]}^{ - m}} - \beta } \right]{a_k}\,,\,\,\,(k \ge j + 1\,\,,n \in {N_0})
\end{align*}
and\  ${\mu _j} = 1 - \sum\limits_{k = j + 1}^\infty  {{\mu _k}}$.\\
We can see that $f(z)$ can be expressed in the form of (2.11).\\

\end{proof}
\begin{theorem}
Let the function $f(z)$ be defined by (1.3) be in the class $T(n,m,\beta )$  then $f(z)$ is close to convex of order $\rho (0\le \rho <1)$ in $\left|z\right|<r_1$, where\\
$ r_1 =r_1(n,m,\beta,\rho)$
\begin{align}
= \mathop {\inf }\limits_k {\left[ {\left( {\frac{{1 - \rho }}{k}} \right)\left\{ {{{\left[ k \right]}^{ - n}}\left[ {\left( {\beta  + 1} \right){{\left[ k \right]}^{ - m}} - \beta } \right]} \right\}} \right]^{\frac{1}{{k - 1}}}}
\end{align}
The result is sharp with the extremal function  $f(z)$ given by (2.3). 
\end{theorem}
\begin{proof}
We must show that $\left| {{f^{'}}\left( z \right) - 1} \right| \le \left( {1 - \rho } \right)$ for $\left|z\right|<r_{1}(n,m,\beta,\rho)$.
Indeed we find from (1.3) that
\begin{align*}
\left| {{f^{'}}\left( z \right) - 1} \right| \le \sum\limits_{k = j + 1}^\infty  {k\,{a_k}\,{{\left| z \right|}^{k - 1}}}.
\end{align*}
Thus $\left| {{f^{'}}\left( z \right) - 1} \right| \le \left( {1 - \rho } \right)$,
\begin{align}
if  \quad  \sum\limits_{k = j + 1}^\infty  {\frac{k}{{1 - \rho }}} \,{a_k}\,{\left| z \right|^{k - 1}} \le 1.
\end{align}
But by Theorem 2.1, equation (2.13) will be true if 
\begin{align*}
\frac{k}{{1 - \rho }}{\left| z \right|^{k - 1}} \le {\left[ k \right]^{ - n}}\left[ {\left( {\beta  + 1} \right){{\left[ k \right]}^{ - m}} - \beta } \right]
\end{align*}
that is, if 
\begin{equation}
\left| z \right| \le {\left[ {\left( {\frac{{1 - \rho }}{k}} \right)\left\{ {{{\left[ k \right]}^{ - n}}\left[ {\left( {\beta  + 1} \right){{\left[ k \right]}^{ - m}} - \beta } \right]} \right\}} \right]^{\frac{1}{{k - 1}}}}
\end{equation}
Theorem 2.7 follows easily from (2.14).
\end{proof}
\begin{theorem}
Let the function $f(z)$ be defined by (1.3) be in the class $T(n,m,\beta)$  then $f(z)$ is starlike of order $\rho (0\le \rho <1)$ in $\left|z\right|<r_2$, where\\
$ r_2 =r_2(n,m,\rho)$
\begin{align}
= \mathop {\inf }\limits_k {\left[ {\left( {\frac{{1 - \rho }}{k-\rho}} \right)\left\{ {{{\left[ k \right]}^{ - n}}\left[ {\left( {\beta  + 1} \right){{\left[ k \right]}^{ - m}} - \beta } \right]} \right\}} \right]^{\frac{1}{{k - 1}}}}
\end{align}
The result is sharp with the extremal $f(z)$ given by equation (2.3).
\end{theorem}
\begin{proof}
It is sufficient to show that 
\begin{align*}
\left| {\frac{{z{f^{'}}(z)}}{{f\left( z \right)}} - 1} \right| \le \left( {1 - \rho } \right)
\end{align*}
for $\left| z \right| < {r_2}\left( {n,m,\rho,\mu } \right)$.Indeed, we find again from Theorem 2.1 that 
\begin{align}
\sum\limits_{k = j + 1}^\infty  {\frac{{k - \rho }}{{1 - \rho }}} \,{a_k}\,{\left| z \right|^{k - 1}} \le 1
\end{align} 
But, by Theorem 2.1, equation (2.16) will be true if 
\begin{align*}
\frac{k-\rho}{{1 - \rho }}{\left| z \right|^{k - 1}} \le {\left[ k \right]^{ - n}}\left[ {\left( {\beta  + 1} \right){{\left[ k \right]}^{ - m}} - \beta } \right]
\end{align*}
that is, if 
\begin{equation}
\left| z \right| \le {\left[ {\left( {\frac{{1 - \rho }}{k-\rho}} \right)\left\{ {{{\left[ k \right]}^{ - n}}\left[ {\left( {\beta  + 1} \right){{\left[ k \right]}^{ - m}} - \beta } \right]} \right\}} \right]^{\frac{1}{{k - 1}}}}
\end{equation}
Theorem 2.8 follows from equation (2.17).
\end{proof}
\begin{theorem}
Let the function $f(z)$ be defined by (1.3) be in the class $T(n,m,\beta)$
then $f(z)$ is convex of order $\rho (0\le \rho <1)$ in $\left|z\right|<r_3$, where\\ 
$ r_3 =r_3(n,m,\beta,\rho)$
\begin{align*}
= \mathop {\inf }\limits_k{\left[ {\left( {\frac{{1 - \rho }}{{k\left( {k - \rho } \right)}}} \right)\left\{ {{{\left[ k \right]}^{ - n}}\left[ {\left( {\beta  + 1} \right){{\left[ k \right]}^{ - m}} - \beta } \right]} \right\}} \right]^{\frac{1}{{k - 1}}}}
\end{align*}
 The result is sharp with the extremal function given by equation (2.3).
\end{theorem}
\begin{proof}
The proof of above theorem is similar to that of Theorem 2.9. Therefore we omit the details involved.
\end{proof}
\begin{theorem}
Let the function $f(z)$ be defined by (1.3) be in the class $T(n,m,\beta)$  and let $c$ be a real number such that $c>-1$.Then the function $G(z)$ defined by
\begin{equation}
G(z) = z - \int_0^z {{t^{c - 1}}\,\,f(t)\,dt\,,(c >  - 1)}
\end{equation}
also belongs to the class $T(n,m,j,\beta)$.
\end{theorem}
\begin{proof}
From the equation (2.18), it follows that
$G(z) = z - \sum\limits_{k = j + 1}^\infty  {{b_k}\,{z^k}}$ \ where \ 
${b_k} = \left( {\frac{{c + 1}}{{c + k}}} \right)\,{a_k}$.\\
Therefore, we have \\
\begin{align*}
\sum\limits_{k = j + 1}^\infty  {{{\left[ k \right]}^{ - n}}\left[ {\left( {\beta  + 1} \right){{\left[ k \right]}^{ - m}} - \beta } \right]{b_k} \le } \sum\limits_{k = j + 1}^\infty  {{{\left[ k \right]}^{ - n}}\left[ {\left( {\beta  + 1} \right){{\left[ k \right]}^{ - m}} - \beta } \right]{a_k} \le } 1.
\end{align*}
Since $ f(z) \in T(n,m,\beta)$.\\
Hence, by Theorem 2.1, $ G(z) \in T(n,m,\beta)$.
\begin{theorem}
Let the function $f(z)$ be defined by (1.3) be in the class $T(n,m,\beta)$ and $c$ be the real number such that $c>-1$. Then function $G(z)$ given by (2.18) is univalent in $\left|z\right|<P^*$ where
\begin{align}
{P^ * } = \mathop {\inf }\limits_k {\left[ {\frac{{\left( {c + 1} \right){{\left[ k \right]}^{ - n}}\left[ {\left( {\beta  + 1} \right){{\left[ k \right]}^{ - m}} - \beta } \right]}}{{\left( {c + k} \right)}}} \right]^{\frac{1}{{k - 1}}}},\,\,\quad (k \ge j + 1).
\end{align}
The result is sharp.
\end{theorem}
\begin{proof}
From the equation (2.18), we have 
\begin{align*}
f(z) = \frac{{{z^{1 - c}}{{\left[ {{z^c}G(z)} \right]}^{'}}}}{{c + 1}} = z - \sum\limits_{k = j + 1}^\infty  {\frac{{c + k}}{{c + 1}}} \,{a_k}\,{z^k}.
\end{align*}
In order to obtain required result, it suffices to show that
$\left| {{G^{'}}(z) - 1} \right| < 1$, whenever $\left| z \right| < {P^ * }$, where
$P^*$ is given by the equation (2.19).\\
Now,
\begin{align*}
\left| {{G^{'}}(z) - 1} \right| \le \sum\limits_{k = j + 1}^\infty  {\frac{{k\left( {c + k} \right)}}{{c + 1}}} \,{a_k}{\left| z \right|^{k - 1}}.
\end{align*}
Thus, 
\begin{equation}
\left| {{G^{'}}(z) - 1} \right| <1 \  \ if \    \  \sum\limits_{k = j + 1}^\infty  {\frac{{k\left( {c + k} \right)}}{{c + 1}}} \,{a_k}{\left| z \right|^{k - 1}} \le 1.
\end{equation}
But Theorem 2.1 confirms that
\begin{align}
\sum\limits_{k = j + 1}^\infty  {{{\left[ k \right]}^{ - n}}\left[ {\left( {\beta  + 1} \right){{\left[ k \right]}^{ - m}} - \beta } \right] \le } 1.
\end{align}
Thus,
\begin{align*}
\frac{{k(c + k)}}{{c + 1}}{\left| z \right|^{k - 1}} < {\left[ k \right]^{ - n}}\left[ {\left( {\beta  + 1} \right){{\left[ k \right]}^{ - m}} - \beta } \right].
\end{align*}
That is, if
\begin{align}
\left| z \right| < {\left[ {\frac{{c + 1}}{{k(c + k)}}{{\left[ k \right]}^{ - n}}\left[ {\left( {\beta  + 1} \right){{\left[ k \right]}^{ - m}} - \beta } \right]} \right]^{\frac{1}{{k - 1}}}}.
\end{align}
Therefore the function given by (2.18) is univalent in $\left| z \right| < {P^ * }$.\\
Let the function ${f_v}(z)\,,\left( {v = 1,2} \right)$ be defined by (2.4). The modified Hadamard product of  ${f_1}(z)$ and ${f_2}(z)$ is defined by
\begin{equation}
\left( {{f_1} * {f_2}} \right)(z) = z - \sum\limits_{k = j + 1}^\infty  {{a_{k,1}}} \  {a_{k,2}}\ {z^k}.
\end{equation}
\end{proof}
\begin{theorem}
Let each of the function ${f_v}(z),\,\,\left( {v = 1,2} \right)$ defined by (2.4) 
be in the class $T(n,m,\beta)$.Then $ {{f_1} * {f_2}} (z) \in  T(n,m,\beta)$ where
\begin{equation}
\gamma  = \frac{{{{\left[ {j + 1} \right]}^{ - n}}{{\left[ {\left( {\beta  + 1} \right){{\left[ {j + 1} \right]}^{ - m}} - \beta } \right]}^2} - {{\left[ {j + 1} \right]}^{ - m}}}}{{{{\left[ {j + 1} \right]}^{ - m}} - 1}}.
\end{equation}
The result is sharp.
\end{theorem}
\begin{proof}
Employing the techniques used by Schild and Silverman \cite{B9}, we need to find largest
$\gamma=\gamma(n,m,\beta)$ such that
\begin{align*}
\sum\limits_{k = j + 1}^\infty  {{{\left[ k \right]}^{ - n}}\left[ {\left( {\beta  + 1} \right){{\left[ k \right]}^{ - m}} - \beta } \right]{a_{k,1}}{a_{k,2}} \le 1} .
\end{align*}
Since
\begin{align*}
\sum\limits_{k = j + 1}^\infty  {{{\left[ k \right]}^{ - n}}\left[ {\left( {\beta  + 1} \right){{\left[ k \right]}^{ - m}} - \beta } \right]{a_{k,1}} \le 1} 
\end{align*}
and 
\begin{align*}
\sum\limits_{k = j + 1}^\infty  {{{\left[ k \right]}^{ - n}}\left[ {\left( {\beta  + 1} \right){{\left[ k \right]}^{ - m}} - \beta } \right]{a_{k,2}} \le 1} .
\end{align*}
By the Cauchy-Schwarz inequality, we have
\begin{align*}
\sum\limits_{k = j + 1}^\infty  {{{\left[ k \right]}^{ - n}}\left[ {\left( {\beta  + 1} \right){{\left[ k \right]}^{ - m}} - \beta } \right]\sqrt {{a_{k,1}}{a_{k,2}}}  \le 1} .
\end{align*}
and thus it is sufficient to show that 
\begin{align*}
{\left[ k \right]^{ - n}}\left[ {\left( {\beta  + 1} \right){{\left[ k \right]}^{ - m}} - \beta } \right]{a_{k,1}}{a_{k,2}} \le {\left[ k \right]^{ - n}}\left[ {\left( {\beta  + 1} \right){{\left[ k \right]}^{ - m}} - \beta } \right]\sqrt {{a_{k,1}}{a_{k,2}}} 
\end{align*}
That is, 
\begin{align*}
\sqrt {{a_{k,1}}{a_{k,2}}}  \le \frac{{\left[ {\left( {\beta  + 1} \right){{\left[ k \right]}^{ - m}} - \beta } \right]}}{{\left[ {\left( {\gamma  + 1} \right){{\left[ k \right]}^{ - m}} - \gamma } \right]}}.
\end{align*}
Note that, 
\begin{align*}
\sqrt {{a_{k,1}}{a_{k,2}}}  \le \frac{1}{{{{\left[ k \right]}^{ - n}}\left[ {\left( {\beta  + 1} \right){{\left[ k \right]}^{ - m}} - \beta } \right]}}.
\end{align*}
Consequently, we need only to prove that
\begin{align*}
\frac{1}{{{{\left[ k \right]}^{ - n}}\left[ {\left( {\beta  + 1} \right){{\left[ k \right]}^{ - m}} - \beta } \right]}} \le \frac{{\left[ {\left( {\beta  + 1} \right){{\left[ k \right]}^{ - m}} - \beta } \right]}}{{\left[ {\left( {\gamma  + 1} \right){{\left[ k \right]}^{ - m}} - \gamma } \right]}}
\end{align*}
Or, equivalently that 
\begin{align*}
\gamma \left[ {{{\left[ k \right]}^{ - n}} - 1} \right] + {\left[ k \right]^{ - m}} \le {\left[ k \right]^{ - n}}{\left[ {\left( {\beta  + 1} \right){{\left[ k \right]}^{ - m}} - \beta } \right]^2}
\end{align*}
\begin{align}
\gamma  = \frac{{{{\left[ k \right]}^{ - n}}{{\left[ {\left( {\beta  + 1} \right){{\left[ k \right]}^{ - m}} - \beta } \right]}^2} - {{\left[ k \right]}^{ - m}}}}{{{{\left[ k \right]}^{ - m}} - 1}}.
\end{align}
Since right hand side of the equation (2.25) is an increasing function of $k$, letting $k=j+1$ in the equation (2.25), we have
\begin{align*}
\gamma  = \frac{{{{\left[ {j + 1} \right]}^{ - n}}{{\left[ {\left( {\beta  + 1} \right){{\left[ {j + 1} \right]}^{ - m}} - \beta } \right]}^2} - {{\left[ {j + 1} \right]}^{ - m}}}}{{{{\left[ {j + 1} \right]}^{ - m}} - 1}}.
\end{align*}
which proves the main assertion of Theorem 2.12. Finally, by taking the function
\begin{equation}
{f_v}(z) = z - \frac{1}{{{{\left[ {j + 1} \right]}^{ - n}}\left[ {\left( {\beta  + 1} \right){{\left[ {j + 1} \right]}^{ - m}} - \beta } \right]}}{z^{j + 1}}
\end{equation}
we can see that result is sharp.
\end{proof}
\begin{theorem}
Let $f_1 (z) \in T(n,m,\beta)$ and $f_2 (z) \in T(n,m,\eta)$.
Then ${f_1} * {f_2}\,(z) \in T(n,m,\xi )$ where\\
 $\xi =\xi (n,m,\eta)$
\begin{align}
= \frac{{{{\left[ {j + 1} \right]}^{ - n}}\left[ {\left( {\beta  + 1} \right){{\left[ {j + 1} \right]}^{ - m}} - \beta } \right]{{\left[ {j + 1} \right]}^{ - n}}\left[ {\left( {\eta  + 1} \right){{\left[ {j + 1} \right]}^{ - m}} - \eta } \right] - {{\left[ {j + 1} \right]}^{ - n}}}}{{{{\left[ {j + 1} \right]}^{ - n}} - 1}}.
\end{align}
The result is best possible for the function 
\begin{align*}
{f_1}(z) = z - \frac{1}{{{{\left[ {j + 1} \right]}^{ - n}}\left[ {\left( {\beta  + 1} \right){{\left[ {j + 1} \right]}^{ - m}} - \beta } \right]}}{z^{j + 1}}
\end{align*}
and
\begin{align*}
{f_2}(z) = z - \frac{1}{{{{\left[ {j + 1} \right]}^{ - n}}\left[ {\left( {\eta  + 1} \right){{\left[ {j + 1} \right]}^{ - m}} - \eta } \right]}}{z^{j + 1}}.
\end{align*}.
\end{theorem}
\begin{proof}
Proceeding as in the proof of Theorem 2.12, we obtain
\begin{align}
\xi  \le \frac{{{{\left[ k \right]}^{ - n}}\left[ {\left( {\beta  + 1} \right){{\left[ k \right]}^{ - m}} - \beta } \right]\left[ {\left( {\eta  + 1} \right){{\left[ k \right]}^{ - m}} - \eta } \right] - {{\left[ k \right]}^{ - m}}}}{{{{\left[ k \right]}^{ - m}} - 1}}.
\end{align}
Since the right hand side of the equation (2.28) is an increasing function of $k$, setting $k=2$ in (2.28), we obtain (2.27).\\
This completes the proof of Theorem 2.13.
\end{proof}
\begin{corollary}
Let the function $f_u (z)$ defined by 
\begin{align}
{f_v}(z) = z - \sum\limits_{k = j + 1}^\infty  {{a_{k,u}}\,{z^k}\,,\,\,({a_{k,v}} \ge 0\,,\,v = 1,2,3)} 
\end{align}
be in the class $T(n,m,\beta)$ and
$\left( {{f_1} * {f_2} * {f_3}} \right)\,(z) \in T(n,m,\delta)$,\\
where
\begin{align}
\delta  = \frac{{{{\left[ {j + 1} \right]}^{ - 2n}}{{\left[ {\left( {\beta  + 1} \right){{\left[ {j + 1} \right]}^{ - m}} - \beta } \right]}^2} - {{\left[ {j + 1} \right]}^{ - m}}}}{{{{\left[ {j + 1} \right]}^{ - m}} - 1}}.
\end{align}
The result is best possible for the functions
\begin{align*}
{f_v}(z) = z - \frac{1}{{{{\left[ {j + 1} \right]}^{ - n}}\left[ {\left( {\beta  + 1} \right){{\left[ {j + 1} \right]}^{ - m}} - \beta } \right]}}{z^{j + 1}}.
\end{align*}
\end{corollary}
\begin{proof}
From Theorem 2.13, we have $(v=1,2,3)$, $\left( {{f_1} * {f_2}} \right)\,(z) \in T(n,m,\gamma)$ where $\gamma$ is given by (2.24).
Now, using Theorem 2.14, we get \\
$\left( {{f_1} * {f_2} * {f_3}} \right)\,(z) \in T(n,m,\delta)$ where $\delta$ is given by (2.30).\\
This completes the proof of corollary.
\end{proof}
\begin{theorem}
Let the function $f_v (z)$ (v=1,2) defined by (2.4) be in the class  $T(n,m,\delta)$, then the function
\begin{align}
g(z) = z - \sum\limits_{k = j + 1}^\infty  {\left( {a_{k,1}^2 + a_{k,2}^2} \right)} \,{z^k}
\end{align}
belongs to the class $T(n,m,\alpha)$ where 
\begin{align}
\alpha  = \alpha (n,m,\alpha) = \frac{{{{\left[ {j + 1} \right]}^{ - n}}{{\left[ {\left( {\beta  + 1} \right){{\left[ {j + 1} \right]}^{ - m}} - \beta } \right]}^2}{{\left[ {j + 1} \right]}^{ - n}} - 2{{\left[ {j + 1} \right]}^{ - m}}}}{{2{{\left[ {j + 1} \right]}^{ - m}} - 1}}.
\end{align}
The result is sharp for the function defined by (2.26).
\end{theorem}
\begin{proof}
By virtue of Theorem 2.1, we have

\begin{align}
\sum\limits_{k = j + 1}^\infty  {{{\left[ k \right]}^{ - n}}\left[ {\left( {\beta  + 1} \right){{\left[ k \right]}^{ - m}} - \beta } \right]} \,{a^2}_{k,1} \le {\left[ {\sum\limits_{k = j + 1}^\infty  {{{\left[ k \right]}^{ - n}}\left[ {\left( {\beta  + 1} \right){{\left[ k \right]}^{ - m}} - \beta } \right]} \,{a_{k,1}}} \right]^2} \le 1
\end{align}
and

\begin{equation}
\sum\limits_{k = j + 1}^\infty  {{{\left[ k \right]}^{ - n}}\left[ {\left( {\beta  + 1} \right){{\left[ k \right]}^{ - m}} - \beta } \right]} \,{a^2}_{k,2} \le {\left[ {\sum\limits_{k = j + 1}^\infty  {{{\left[ k \right]}^{ - n}}\left[ {\left( {\beta  + 1} \right){{\left[ k \right]}^{ - m}} - \beta } \right]} \,{a_{k,2}}} \right]^2} \le 1.
\end{equation}
It follows from (2.33) and (2.34) that 
\begin{align}
{\left[ {\sum\limits_{k = j + 1}^\infty  {\frac{1}{2}{{\left[ k \right]}^{ - n}}\left[ {\left( {\beta  + 1} \right){{\left[ k \right]}^{ - m}} - \beta } \right]} \,} \right]^2}\left( {{a^2}_{k,1} + {a^2}_{k,2}} \right) \le 1.  
\end{align}
Therefore, we need to find the largest $\alpha$ such that 
\begin{align*}
{\left[ k \right]^{ - n}}\left\{ {\left( {\alpha  + 1} \right){{\left[ k \right]}^{ - m}} - \alpha } \right\} \le \frac{1}{2}{\left[ {{{\left[ k \right]}^{ - n}}\left\{ {\left( {\beta  + 1} \right){{\left[ k \right]}^{ - m}} - \beta } \right\}} \right]^2}.
\end{align*}
That is, 
\begin{align}
\alpha  \le \frac{{{{\left[ {j + 1} \right]}^{ - m}}{{\left[ {\left( {\beta  + 1} \right){{\left[ {j + 1} \right]}^{ - m}} - \beta } \right]}^2} - 2{{\left[ {j + 1} \right]}^{ - m}}}}{{2\left[ {{{\left[ {j + 1} \right]}^{ - m}} - 1} \right]}}.
\end{align}
Since right hand side of (2.36) is an increasing function of $k$, we readily have
 \begin{align*}
\alpha  = \frac{{{{\left[ {j + 1} \right]}^{ - m}}{{\left[ {\left( {\beta  + 1} \right){{\left[ {j + 1} \right]}^{ - m}} - \beta } \right]}^2} - 2{{\left[ {j + 1} \right]}^{ - m}}}}{{2\left[ {{{\left[ {j + 1} \right]}^{ - m}} - 1} \right]}}.
\end{align*}
Hence proof of Theorem 2.14 is complete.
\end{proof}
\end{proof}\\
\\
\\
\textbf { Author Declaration}:\\
We wish to confirm that here are no known conflicts of interest associated with this publication and there has been no significant financial support for this work.

\pagebreak

\end{document}